\documentstyle[12pt]{article}

\def\sym{\fam\comfam\com}
\font\tensym=msbm10 at 12pt
\font\sevensym=msbm7
\font\fivesym=msbm5

\newfam\symfam
\textfont\symfam=\tensym
\scriptfont\symfam=\sevensym
\scriptscriptfont\symfam=\fivesym
\def\sym{\fam\symfam\relax}
\newcommand{\ds}{\displaystyle}

\def\Z{{\sym Z}}

\def\R{{\sym R}}
\def\C{{\sym C}}
\def \T{{\sym T}}
\def \P{{\sym P}}

\newtheorem{thm}{Theorem}[section]

\newtheorem{lemm}[thm]{Lemma}
\newtheorem{defn}[thm]{Definition}

\newtheorem{que}[thm]{Question}

\newtheorem{Cor}[thm]{Corollary}

\newenvironment{proof}{{\bf Proof:  }}{\hfill\rule{2mm}{2mm}}

\date{ 19 November 1999}
\author{\MakeLowercase{e.} H. \MakeLowercase{el} A\MakeLowercase{BDALAOUI} \\
\footnotesize Department of Mathematics, University of Rouen Normandy \\
\footnotesize CNRS, LMRS UMR 60 85,\\
\footnotesize Saint-Etienne de Rouvray \\
\footnotesize e-mail : elHoucein.elAbdalaoui@univ-rouen.fr\\
}
\title{ van der Corput inequality for real line  and Wiener-Wintner theorem  for amenable groups.}
\begin{document}

\maketitle
 
{\bf Abstract. }{ We extend the classical van der Corput inequality  to the real line. As a consequence, we obtain a simple proof of the Wiener-Wintner theorem for the $\R$-action which assert that for any family of maps $(T_t)_{t \in \R}$  acting on
the Lebesgue measure space $(\Omega,{\cal {A}},\mu)$ where $\mu$ is a probability measure and for any $t\in \R$, $T_t$ is measure-preserving transformation on measure space $(\Omega,{\cal {A}},\mu)$ with
 $T_t \circ T_s =T_{t+s}$, for any $t,s\in \R$. Then, for any 
$f \in L^1(\mu)$, there is a a single null set  off which
$\displaystyle \lim_{T \rightarrow +\infty} \frac1{T}\int_{0}^{T} f(T_t\omega) e^{2 i \pi \theta t} dt$
exists for all $\theta \in \R$. We further present the joining proof of the amenable group version of Wiener-Wintner theorem due to Weiss and Ornstein.}

\section{Introduction}
\noindent{}In this paper, using our generalization of van der Corput inequality for the real line, we present a simple proof of  Wiener-Wintner theorem for the flows. We further present the joining proof of the amenable groups version of it due to Ornstein and Weiss \cite{OW}. This accomplished by applying the Furstenberg joinings machinery. The classical Wiener-Wintner theorem \cite{W-W} assert the following. \\

\noindent{\bf Theorem.} {\it Let $(\Omega,{\cal {A}},\mu,T)$ be a dynamical
system where $\mu$ is a probability measure. Then, for any $f$ in $L^1(\mu)$, There is a
set $\Omega'$ of full measure such that for any $\omega \in \Omega'$ the sums
$$\frac{1}{N}\sum_{0}^{N-1} f(T^n\omega)z^n$$
converge for all $z$ in the unit circle $C=\{z \in {\C}~:~|z|=1\}$}.\\ 

The original proof can be found in \cite{W-W}. 
Subsequently, Furstenberg in \cite{F1} obtain a joining proof of Wiener-Wintner 
theorem. Later, I. Assani \cite{A}, A.${~}$Below \& V. Losert \cite{B-L} proved the stronger version  of this theorem. This stronger version is due to Bourgain \cite{B}. Theirs proofs is based on the Hellinger integral 
(Known also as affinity principale). In \cite{L1}, E. Lesigne generalize Wiener-Wintner theorem 
to the polynomial case.  His proof is based on the Furstenberg's joinings 
technique. Afterwards, in \cite{L2}, using van${~}$der${~}$Corput inequality and the spectral theory of skew products, 
he extended the stronger version of polynomial Wiener-Wintner 
theorem to the case of weak-wixing dynamical systems\footnote{Seven year after this note was written ,  M. Lacey and  E. Terwilleger \cite{LT} produce an oscillation proof of the Hilbert version of Wiener-Wintner theorem.}. 
 \\

In this paper, we extend van der Corput inequality to the continuous time and we give a simple proof of the flow version of Wiener-Wintner theorem. We further present the Ornstein-Weiss's joining of the amenable group version of this fundamental theorem in ergodic theory.  The proof is based on Furstenberg's joinings machinery combined with the recent result of E. Lindenstrauss \cite{Li}. \\

The plan of the paper is as follows. In Section \ref{vdCR}, we state and prove the continuous van der Corput inequality and the flow version of Wiener-Wintner theorem. In section \ref{Jww}, we state and prove the amenable group version of Wiener-Wintner theorem.

\section{van der Corput for real line} \label{vdCR}
In this section, we state our first main result.\\

\begin{thm}[{\bf {Theorem (van der Corput).}}]\label{vdrc1}
 Let $(u(t))_{t \in [0,T]}$ be an integrable complex valued function and let $S \in (0,T]$. Then 
\begin{eqnarray}
\Big|\displaystyle \int_{0}^{T}u(t) dt\Big|^2 \leq 
\frac{S+T}{S^2}\displaystyle\int_{0}^{S} \displaystyle\int_{0}^{S} \int_{0}^{T}u(t+s'-s) \overline{u}(t) ds ds' dt.
\end{eqnarray}
\end{thm}
\begin{proof}We start by noticing that we have
\begin{eqnarray}
	S\ds \int_{0}^{T}u(t) dt= \ds \int_{0}^{T+S}\ds \int_{0}^{S}\tilde{u}(t-s) ds dt,
\end{eqnarray}
where $\tilde{u}$ stand for 
$$
\tilde{u}(t)=\left\{\begin{array}{lll}
0  &\textrm{if}\;  t \leq 0,\\
u(t)  &\textrm{if} \;0 \leq  t \leq T,\\
0   &\textrm{if~not.} 
\end{array}
\right.
$$
Indeed, we have
\begin{eqnarray}
\int_{0}^{T+S}\int_{0}^{S}\tilde{u}(t-s) ds dt&=\ds \int_{0}^{S}\ds \int_{-s}^{T+S-s} \tilde{u}(t) dt ds\\
&=\ds \int_{0}^{S}\ds \int_{0}^{T} u(t) dt ds\\
&=S\ds \int_{0}^{T} u(t) dt.
\end{eqnarray}
Whence, 
\begin{eqnarray}
S^2\Big|\ds \int_{0}^{T}u(t) dt\Big|^2= \Big|\ds \int_{0}^{T+S}\ds \int_{0}^{S}\tilde{u}(t-s) ds dt\Big|^2.
\end{eqnarray}
Now, applying Cauchy-Schwarz inequality, we obtain
\begin{eqnarray}
S^2\Big|\ds \int_{0}^{T}u(t) dt\Big|^2 
\leq (T+S) \Big(\int_{0}^{T+S}\Big|\ds \int_{0}^{S}\tilde{u}(t-s) ds\Big|^2 dt\Big).
\end{eqnarray}
But 
\begin{eqnarray}
	\Big|\ds \int_{0}^{S}\tilde{u}(t-s) ds\Big|^2 &=
	\ds \int_{0}^{S} \tilde{u}(t-s) \overline{\tilde{u}}(t-s') ds ds'\\
&=\ds \int_{0}^{S} \tilde{u}(t-s) \overline{\tilde{u}}(t-s') ds ds'\\
&=\ds \int_{0}^{S} \tilde{u}(t+s'-s) \overline{\tilde{u}}(t) ds ds'
\end{eqnarray}
Whence
\begin{eqnarray}
\Big|\displaystyle \int_{0}^{T}u(t) dt\Big|^2 \leq 
\frac{S+T}{S^2}\displaystyle\int_{0}^{S} \displaystyle\int_{0}^{S} \int_{0}^{T}u(t+s'-s) \overline{u}(t) ds ds' dt.
\end{eqnarray}
This achieve the proof of the theorem. 
\end{proof}
\begin{thm}[Limit version of continuous van der Corput theorem.]\label{vdrc2}  Let $(u(t))_{t \in \R}$ be a bounded complex valued function. Then 
\begin{eqnarray}
	&\ds \limsup_{T \rightarrow \infty}\Big|\frac{1}{T}\ds\int_{0}^{T}u(t) dt \Big|^2 \nonumber\\
	&\leq 
	\ds \limsup_{S \rightarrow \infty}\frac{1}{S^2}\ds \int_{0}^{S}\ds \int_{0}^{S} 
		\ds \limsup_{T \rightarrow \infty} \ds \frac{1}{T}\int_{0}^{T}u(t+s'-s) \overline{u}(t)  ds ds' dt.
\end{eqnarray}
\end{thm}
\begin{proof}Straightforward from Theorem \ref{vdrc1}.
\end{proof}\\

\noindent Now, let us state the continuous version of Wiener-Wintner theorem.

\begin{thm}[{\bf {Continuous version of  Wiener-Wintner theorem.}}] \label{WWt}
	
	Let $(T_t)_{t \in \R}$ be a maps acting on the Lebesgue measure space $(\Omega,{\cal {A}},\mu)$ where $\mu$ is a probability measure and for any $t\in \R$, $T_t$ is measure-preserving transformation on measure space $(\Omega,{\cal {A}},\mu)$ with
	$T_t \circ T_s =T_{t+s}$, for any $t,s \in \R$. Then, for any 
	$f \in L^1(\mu)$, there is a a single null set off which
	$\displaystyle \lim_{T \rightarrow +\infty} \frac1{T}\int_{0}^{T} f(T_t\omega) e^{2 i \pi \theta t} dt$
	exists for all $\theta \in \R$. 
\end{thm}

We will assume without loss of generality  that $\mu$ ergodic. Indeed, otherwise, on can use the ergodic decomposition of $\mu$. So, it is sufficient to prove the following : \\
 
 \begin{thm}\label{Sww}
  For any $f$ in $L^2(\mu)$, there is a set 
 $\Omega'$ of full measure such that the sums 
 $\displaystyle \lim_{T \rightarrow +\infty} \frac1{T}\int_{0}^{T} f(T_t\omega) e^{2 i \pi \theta t} dt$
 converge to 0 for all $\theta$ in $\R$, where $e^{2 \pi i \theta} \not \in
 e(T)$ and $\omega \in \Omega'$. $e(T)$ stand for the set of eigenvalue of the Koopman operator $U_T\; :\; g \mapsto g \circ T.$
 \end{thm}
\noindent{}Before proceeding to the proof of Theorem \ref{Sww}, let us notice  that it suffices to prove it for a dense class of functions ($L^2$ functions for instance). Indeed,
Put $$R(\omega,f)=\limsup_{T \longrightarrow +\infty}\Big|\ds \int_{0}^{T}f(T_t(\omega)) e^{2 \pi i t \theta}dt\Big|,$$
and assume that $g$ in the dense class for which theorem holds. Then 
$$R(\omega,f)=R(\omega,f-g).$$
and then
$$\mu\{\omega:R(\omega,f-g)>\epsilon\}\leq \frac{||f-g||_1}{\epsilon}.$$
\noindent{}We thus get by the density of $L^2(\mu)$ in $L^1(\mu)$, that there
exist $g$ in $L^2(\mu)$ such that : $||f-g||_1<{\epsilon}^2$. Then
$$\mu\{\omega:R(\omega,f-g)>\epsilon\}\leq\epsilon.$$
\noindent{}Since $\epsilon$ is arbitrary, we see $R(\omega,f)=0$ a.e., where
the null set excluded is independent of $z$.\\

\noindent{}We start by recalling that by Bochner theorem, for any $f \in L^2(X)$, there exists a unique finite Borel measure $\sigma_f$ on $\R$ such that
\[\widehat{\sigma_f}(t)=\ds \int_{\R} e^{-it\xi}\ d\sigma_f(\xi)=\langle U_tf,f \rangle=
\ds \int_{\Omega} f \circ T_t(\omega)\cdot \overline{f}(\omega) \ d\mu(\omega).\]
$\sigma_f$ is called the \emph{spectral measure} of $f$. If $f$ is eigenfunction with eigenfrequency $\lambda$ then
the spectral measure is the Dirac measure at $\lambda$.\\

\noindent{}We need also the following fundamental results.\\

\begin{thm}\label{wwhm} Let  $(\Omega,{\mathcal{A}},\mu,(T_t)_{t \in \R})$ be an \emph{ergodic} dynamical flow. Then,  for any $S>0$ and all $f,g\in  L^2(X)$, for almost all $\omega \in \Omega$, we have
	$$\lim_{\tau \to +\infty}\frac1{\tau}\int_{0}^{\tau}f(T_{t+s}\omega)\cdot g(T_t\omega)\,dt=\int_{\Omega}f\circ T_s\cdot g\,d\mu$$
	uniformly for $s$ in the interval $[-S,S]$.
\end{thm}
\noindent This yields the exact result need it.
\begin{Cor} \label{wwcor}Let $f\in L^2(\mu)$. There exist a full measure subset $\Omega_f$ of $\Omega$ such that, for any $\omega\in \Omega_f$ and any $s\in\R$, we have
	$$
	\lim_{\tau\to\infty} \frac1{\tau}\int_0^{\tau} f(T_{t+s}\omega)\cdot \overline{f}(T_t\omega)\, dt=\int_X f\circ T_s \cdot \overline{f}\,d\mu.
	$$
\end{Cor}

\noindent\begin{proof}[ \textbf{of Theorem} \ref{Sww}] Let $f$ in $L^\infty(\mu)$ and $\omega \in \Omega_f$ as in Corollary \ref{wwcor}, then we have 
 \begin{eqnarray}
 	\lim_{\tau\to\infty} \frac1{\tau}\ds \int_0^{\tau} f(T_{t+s}\omega)\cdot \overline{f}(T_t\omega)\, dt&=\ds \int_X f\circ T_s \cdot \overline{f}\,d\mu\\
 	&\stackrel {\rm def }{=}<f \circ T_s,f>.
 \end{eqnarray}
 Put $$u(t)=f(T_t\omega) e^{2  \pi i t \theta},$$ 
\noindent and apply further van der Corput's 
 inequality (Theorem \ref{vdrc1}) to get
 \begin{eqnarray}
 &\Big|\frac{1}{\tau}\displaystyle \int_{0}^{\tau}f(T_t\omega) e^{2  \pi i t \theta} dt\Big|^2 \nonumber\\&\leq 
 \frac{S+\tau}{\tau S^2}\displaystyle\int_{0}^{S} \displaystyle\int_{0}^{S} e^{2 \pi i (s-s')\theta}\frac{1}{\tau}\int_{0}^{\tau} f(T_{t+s-s'}) \overline{f}(T_t\omega)   dt ds ds'.	
 \end{eqnarray}
 \noindent{}We thus deduce that for almost all $\omega$ and all $\theta \in
 {\R}$, we have 
 \begin{eqnarray}
 &{\limsup_{\tau \rightarrow \infty}} \Big|\frac{1}{\tau}\displaystyle \int_{0}^{\tau}f(T_t\omega) e^{2  \pi i t \theta} dt\Big|^2 \nonumber\\
 &\leq \frac{1}{S^2}\displaystyle\int_{0}^{S} \displaystyle\int_{0}^{S}e^{2 \pi i (s-s')\theta}\Big({\lim_{\tau \rightarrow \infty}}\frac{1}{\tau}\int_{0}^{\tau} f(T_{t+s-s'}) \overline{f}(T_t\omega) dt\Big) ds ds'.	
 \end{eqnarray}

 \noindent{}This combined with Corollary \ref{wwcor} gives
 \begin{eqnarray}
 &{\limsup_{\tau \rightarrow \infty}} \Big|\frac{1}{\tau}\displaystyle \int_{0}^{\tau}f(T_t\omega) e^{2  \pi i t \theta} dt\Big|^2 \nonumber\\
 &\leq \frac{1}{S^2}\displaystyle\int_{0}^{S} \displaystyle\int_{0}^{S} \Big(
 \ds \int_{\R} e^{2 \pi i (s-s')(\theta-\gamma)} d\sigma_f(s-s')\Big) ds ds',	
 \end{eqnarray}
 where $\sigma_f$ stand for the spectral measure of $f$. But, since 
 \begin{eqnarray}
  \frac{1}{S^2} 
 \displaystyle\int_{0}^{S} \displaystyle\int_{0}^{S}  e^{2 \pi i (s-s')(\theta-\gamma)} ds ds'
 =\Big|\frac{1}{S} 
 \displaystyle\int_{0}^{S} e^{2 \pi i s(\theta-\gamma)} ds\Big|^2 
 \end{eqnarray}
 \noindent{}if $\theta \neq \gamma$, we have
\begin{eqnarray}
\lim_{S \rightarrow \infty}\frac{1}{S^2} 
\displaystyle\int_{0}^{S} \displaystyle\int_{0}^{S}  e^{2 \pi i (s-s')(\theta-\gamma)} ds ds'
=0. 
\end{eqnarray}
 \noindent{}Whence, if $e^{2 \pi i \theta}$ is not a eigenvalue of $(T_t)$, we have
 $$ \lim_{S \rightarrow \infty}\frac{1}{S^2}\displaystyle\int_{0}^{S} \displaystyle\int_{0}^{S} \Big(
 \ds \int_{\R} e^{2 \pi i (s-s')(\theta-\gamma)} d\sigma_f(s-s')\Big) ds ds'=0.$$
 \noindent{}Since all the sums are bounded, we deduce from Lebesgue theorem that 
 for almost all $\omega$, and for all $\theta$ in ${\R}$, where $e^{2 \pi i \theta} \not \in
 e(T)$,
 $${\lim_{\tau \rightarrow \infty}}\frac{1}{\tau}\displaystyle \int_{0}^{\tau}f(T_t\omega) e^{2  \pi i t \theta} dt=0,$$
 \noindent{}and this finish the proof of the theorem.
 \end{proof}

\section*{2. Joining's proof of Wiener-Wintner Theorem for action of amenable group}\label{Jww}

In this section, we deal with actions on Lebesgue spaces, that is, given a 
locally compact groupe $G$ and the a Lebesgue space $(X, {\cal {A}}, \mu)$, a 
$G-${\it {action}} is a measurable mapping $G \times X \rightarrow X$, 
$(g,x) \mapsto g.x$, such that for all $g,h \in G$, $g.(h.x)=(gh).x$ and 
$e.x=x$ for almost all $x \in X$ ( where $e$ is the identity in $G$). 
Furthermore, $T_g ~:~ x \mapsto g.x$ is measure -preserving for every $g \in G$. 
We will mainly concerned with $G$ which is amenable group ( locally compact
second countable) or the subclass of locally compact abelian groups. \\
We recall that  $G$ is an amenable group if for
any compact $K \subset G$ and $\delta > 0$ there is a compact set $F \subset G$ such that
\begin{eqnarray}\label{Folner}
\textrm{h}_L(F \Delta KF) < \delta \textrm{h}_L (F),
\end{eqnarray}
where $\textrm{h}_L$ stand for the left Haar measure on $G$. It is well known that the amenability is equivalent to the existence of F\o{}lner sequence $(F_n)$, that is, $(F_n)$ is a sequence of
compact subsets of $G$ for which for every
compact $K$ and $\delta > 0$, for all large enough $n$ we have that $F_n$ satisfy (\ref{Folner}). Assume further that $(F_n)$ satisfy the so-called Shulman Condition ,that is, for some $C > 0$ and all $n$
\begin{eqnarray}\label{Shulman}
\textrm{h}_L\Big(\bigcup_{k \leq n }F_k^{-1}F_n\Big) \leq C .\textrm{h}_L\Big(F_n\Big). 
\end{eqnarray}
Under this assumptions, E. Lindenstrauss proved that the Birkhoff pointwise ergodic theorem holds, that is, Then for any
$f \in L^1(\mu)$, there is a $G$-invariant $f^*\in L^1 (\mu)$ such that
$$
\lim\frac{1}{\textrm{h}_L\big(F_n\big)}\ds \int_{F_n} f(g\omega) d{\textrm{h}}_L(g) = f^*(\omega) \; \; \textrm{a.e.}.$$

To formulate the $G$-version of Wiener-Wintner theorem, we replace the group rotations by homomorphisms $\Theta$ from $G$ to a finite dimensional unitary group $U_d$. The  canonical action in this case is given by $g.u=\Theta(g).u,$  $u \in U_d$ and $g \in G$. The invariant measure is the Haar measure on $U_d$. In this setting, we formulate the  Wiener-Wintner theorem as follows:
\begin{thm}[Group version of  Wiener-Wintner theorem]\label{Gww}
Let $G$ be an amenable group acting on a Lebesgue space $(\Omega, \mathcal{A}, \mu)$ and assume that $G$ satisfy Shulman condition. Let $f \in L^\infty(\mu)$. Then, there is a set $\Omega_f$ of full measure such for any $\omega \in \Omega_f$
$$
\frac{1}{\textrm{h}_L\big(F_n\big)}\ds \int_{F_n} f(g\omega)\phi(\Theta(a)u) d{\textrm{h}}_L(g)  \; \; $$
converge for all finite dimensional unitary representation $\Theta$ of $G$ into $U_d$ (all $d$), all continuous function $\phi$ on $U_d$ and all $u \in U_d$. We further have that the limit on the orthocomplement of the space spanned by the finite dimensional invariant subspaces is zero.
\end{thm}

Before proceeding to the proof let us recall some important tools.\\
  
A {\it {joining}} of two actions of the same group 
${\cal {X}}=(X,{\cal {A}},\mu,G)$  and  ${\cal {Y}}=(Y,{\cal {B}},\nu,G)$  
is the probability measure $\lambda$ on $(X \times Y,
{\cal {A}} \times {\cal {B}})$ which is invariant under the diagonal action
of $G$ $(g.(x,y)=(g.x,g.y))$ and whose marginals on $({\cal {A}} \times Y)$ and
$(X \times {\cal {B}})$ are $\mu$ and $\nu$ respectively (i.e. if $A \in
{\cal {A}}$, $\lambda(A \times Y)=\mu(A)$; and if $B \in {\cal {B}},
~\lambda(X \times B)=\nu(B)).$ The set of joinings is never empty. (Take $\mu
\times \nu$). As we deal with Lebesgue spaces, a joining $\lambda$ of two
ergodic $G$-actions ${\cal {X}}$ and ${\cal {Y}}$ has the property that there
exists a Lebesgue space $\Omega$ and the probability $\P$ on $\Omega$ such
that $\lambda=\int \lambda_{\omega} d\P(\omega)$, where $\lambda_{\omega}$ is
ergodic. (This is just the ergodic decomposition of $\lambda$, and as the
marginals of $\lambda$ are ergodic a.e., $\lambda_{\omega}$ is joining). 
Therefore the set of ergodic joinings is never empty.{\footnote {see [D-R],
for instance.}}\\

Historically, joinings were introduced by H. Furstenberg in his paper \cite{F2} on
disjointness. In particular, he defined the important notion of disjointness
for $\Z$-{\it {action}} in the following way :  $(X,{\cal {A}},\mu,T)$ and
$(Y,{\cal {B}},\nu,S)$ is {\it disjoint} if the only joining between them is
the product joining. In the case of $G$-{\it {action}} we have the following 
definition.\\

\begin{defn}
Let ${\cal {X}}$  and ${\cal {Y}}$ be 
two actions of the same group $G$.  ${\cal {X}}$  and ${\cal {Y}}$ are {\it disjoint} if the only joining between them is the product joining. We denote this disjointness by ${\cal {X}} \perp {\cal 
	{Y}}$.
\end{defn} 
\noindent In the case of $\Z$-actions, H. Hahn \& W. Parry obtain in \cite{H-P} that if two transformations
have mutually singular maximal types, then they are disjoint. But, 
As for the joinings theory, the spectral theory of $\Z$-actions can be extended to the case 
of locally abelien $G-$actions. Therefore, we have the following group version of Hahn-Parry theorem.
\\

\begin{thm}[(Hahn \& Parry]
If two $G$-actions ${\cal {X}}$ and ${\cal {Y}}$
have mutually singular maximal spectral types, then  they are
disjoint.
\end{thm}

\begin{proof}
Let recall that the spectral measure of a function $f 
\in L^2(X)$ under the operators $U_g$ ( defined on $L^2(X)$ by $U_g(f)=
f\circ T_g$) is the measure $\sigma_f$ on $\displaystyle\stackrel
{\wedge }{G}$ (dual group of $G$, i.e., the set of
all continuous characters of $G$) where its Fourier transform $\displaystyle\stackrel
{\wedge }{\sigma_f}$
is given by $\displaystyle\stackrel{\wedge }{\sigma_f}(g)=<U_gf,f)$. Now, we 
follows the proof given in \cite{Th}. In $X \times Y$ endowed with a joining 
measure $\lambda$, consider $f_1 \in L^2(X)$ and $f_2 \in L^2(Y)$ and 
consider $H_{f_1}$ the $L^2(\lambda)$ closure of the linear span of the
functions $(U_g(f_1)-\int f_1 d \mu ) \times 1_Y$, $g \in G$. The projection of $1_X\times f_2$ on
$H_{f_1}$ will have a spectral measure absolutely continuous with respect to 
the spectral type of $U_g$ on $L^2(X)$ and thus has to be $0$. Therefore 
$1_X\times f_2 \perp (f_1-\int f_1)\times 1_Y$, and $\int f_1(x) f_2(y)
d\lambda(x,y)=\int f_1 d \mu\int f_2d\nu$.
\end{proof} 

\noindent From this theorem we have the following.\\

\begin{Cor}\label{Csdisjoint}
Let $\chi_0$ be a non trivial character and define the
action of $G$ on torus $\T$ by
$(g,e^{ix}) \mapsto \chi_0(g) e^{ix}$. Assume that for any $n \in \Z$,
the character $\chi_0^n$ define on $G$ by $g \mapsto \chi_0(g^n)$ is not
eigenvalue of the $G$-action on ${\cal {X}}$. The the $G$-action on $\T$
and the $G$-action on ${\cal {X}}$ are disjoint.
\end{Cor} 
\begin{proof}
Let recall that $\chi_0$ is a eigenvalue of $G$-
action if there exist a eigenfunction $f \in L^2(X,\mu)$ such that 
$f\circ T_g =\chi_0(g)~f.$ We deduce that the spectral measure of $f$ is 
$||f||_2^2\delta_{\chi_0}$ ($\delta_{\chi_0}$ is the Dirac measure on 
$\chi_0$). Since for any $n \in \Z$, $\chi_0^n$ is not eigenvalue of $G$-action
on ${\cal {X}}$, we conclude that the maximal spectral types of this two $G$-actions 
are mutually singular. Now apply the Hahn-Parry theorem to complete the proof.
\end{proof} \\

\noindent For the general case of amenable group which satisfy Shulman condition, we have the following lemma from \cite{OW}.
\begin{lemm}\label{owlemma} Let $U$ be the closure of $\Theta(G)$ in $U_d$. Then, if the product
	$(U,\Theta,G) \times (\Omega,\mathbf{A}, \mu, G)$ is ergodic then there is only on $G$-invariant measure on $U \times \Omega$ that projects onto $\mu$ on $\Omega$.
\end{lemm}
\begin{proof}[ \textbf{of Theorem \ref{Gww}} ] We start by assuming without lost of generality that the action on $(\Omega,\mathbf{A}, \mu, )$ is ergodic and by presenting the proof for the case when $G$ is locally abelien group. Let $f \in L^\infty(\mu)$ and $\phi$ continuous function. Then, by the pointwise theorem there is a set of full measure of $\omega$. Assume that $\omega$ is in this subset and let $\chi_0 \in \hat{G}$ such $\chi_0$ is not eigenvalue. Then, the product $(U,\Theta,G) \times (\Omega,\mathbf{A}, \mu, G)$ is ergodic. Moreover, by taking a subsequence $(n_k)$, we can assume that 
$$
\lim_{k \longrightarrow +\infty}\frac{1}{\textrm{h}_L\big(F_n\big)}\ds \int_{F_{n_k}} f(g\omega)\phi(\Theta(a)u) d{\textrm{h}}_L(g)=\lambda(f \otimes \phi). $$
It follows that $\lambda$ is a joining and by Corollary \ref{Csdisjoint} $\lambda=dh \times \mu$. We end the proof by noticing that there is a countable of eigenvalue.  The general case follows in the same manner by taking  
$$F(\omega)=\int \psi(u) I(u_1\omega) du,$$
where $I$ is a bounded invariant functions on $U \times \Omega$ and $\psi$ is any positive continuous function on $u$. Therefore, transforming $F$ by $g$ is the same as transforming $\psi$ by $\Theta(g)$. We thus have that a non-constant $I$ will give rise to finite dimensional invariant subspaces for $G$ on $\Omega$. Moreover, by taking $(U,\Theta,G)$ not in in the list of countable representations $(U_j,\Theta_j,G)$, the condition of Lemma \ref{owlemma} is satisfied and therefore as before the only joining is the product measure, and we are done. 
\end{proof}
\begin{que}We ask on the possible extension of van der Corput inequality to the locally compact group and its application to obtain produce a direct proof of the group version of Wiener-Wintner theorem.   
\end{que}
\textbf{Acknowledgment.}

The author would like to express his thanks to Jean-Paul Thouvenot and E. Lesigne for a discussion on the subject. 

\end{document}